\documentclass[12pt,a4paper]{article}
\usepackage{amsfonts}
\usepackage{mathrsfs}
\usepackage{amsmath}
\usepackage{color}
\usepackage{stmaryrd}
\usepackage{amssymb}
\usepackage{bbm}
\usepackage{graphicx}
\usepackage{enumerate}
\usepackage{theorem}
\usepackage{ulem}
\usepackage{authblk}
\makeatletter
\def\tank#1{\protected@xdef\@thanks{\@thanks
        \protect\footnotetext[0]{#1}}}
\def\bigfoot{

    \@footnotetext}
\makeatother

\newcommand{\ea}{\end{array}}
\newtheorem{theorem}{Theorem}[section]
\newtheorem{proposition}{Proposition}[section]

\newtheorem{lemma}{Lemma}[section]
\newtheorem{definition}{Definition}[section]

{\theorembodyfont{\rmfamily}
}

\newenvironment{proof}{Proof.}

\allowdisplaybreaks

\begin{document}
\title {\Large \bf Large Deviation Principle of Stochastic Evolution Equations with reflection} 
\author[1]{Zdzisław Brzeźniak\thanks{
 E-mail:zdzislaw.brzezniak@york.ac.uk}}
\author[2]{Qi Li\thanks{ E-mail:vivien777@mail.ustc.edu.cn}}
\author[3]{Tusheng Zhang\thanks{ E-mail:Tusheng.Zhang@manchester.ac.uk}}

\affil[1]{Department of Mathematics, University of York, Heslington, YO10 5DD, York, United Kingdom}
\affil[2]{School of Mathematical Sciences, University of Science and Technology of China, Hefei, Anhui 230026, China.}
\affil[3]{Department of Mathematics, University of Manchester, Oxford Road,
Manchester, M13 9PL, UK.}
\renewcommand\Authands{ and }
\date{}
\maketitle

\begin{center}
\begin{minipage}{130mm}
{\bf Abstract.}
In this paper, we establish a large deviation principle for stochastic evolution equations with reflection in an infinite dimensional ball. Weak convergence approach plays an important role.

\vspace{3mm} {\bf Keywords.} Large deviations; Stochastic evolution equations; reflection; weak convergence.

\end{minipage}
\end{center}

\section{Introduction}
\setcounter{equation}{0}
 \setcounter{definition}{0}
Let $H$ be a seperable Hilbert space with the norm $|\cdot|$ and inner product $(\cdot,\cdot)$. $D:=B(0,1)$ denotes the open unit ball on $H$. Let $A$ be a self-adjoint, positive define operator on $H$ and let $B$ be a unbounded bilinear map from $H\times H$ to $H$. Let $W$ denote a real-valued Brownian Motion on a filtered probability space $(\Omega,\mathcal{F,\mathbb{F}}=(\mathcal{F}_t)_{t\geq 0},\mathbb{P})$ satisfying the usual conditions. Consider the following stochastic evolution equations(SEEs) with reflection:
\begin{equation}\label{SEE}
\left\{
\begin{aligned}
du(t)&+Au(t)dt =f(u(t))dt +B(u(t),u(t))dt + \sigma(u(t))dW(t)+dL(t), t\geq0,\\
 u(0)& =u_0, \quad u_0\in \overline{D},
\end{aligned}
\right.
\end{equation}
Where $u$ is a $\overline{D}$-valued continuous stochastic process and $L$ ia a $H$-valued continuous stochastic process, which plays the role of a local time. $\overline{D}$ denotes the closed unit ball in the Hilbert space $H$. $f$ and $\sigma$ are measurable mappings.

Stochastic partial differential equations(SPDEs) with reflection can be used to model the evolution of random interfaces near a hard wall, see \cite{FO}. Existence and uniqueness of the above stochastic reflected problems were established in \cite{BZ}. For the study of real-valued SPDEs with reflection we refer the readers to  \cite{NP}, \cite{DP}, \cite{XZ} and references therein.

The purpose of this paper is to establish a  small noise large deviation principle(LDP) for the following  SEEs with reflection:
\begin{equation}\label{1.1}
\left\{
\begin{aligned}
du^\varepsilon(t)&+Au^\varepsilon(t)dt =f(u^\varepsilon(t))dt +B(u^\varepsilon(t),u^\varepsilon(t))dt + \sqrt{\varepsilon}\sigma(u^\varepsilon(t))dW(t)+dL^\varepsilon(t), t\geq0,\\
 u^\varepsilon(0)& =u_0, \quad u_0\in \overline{D}.
\end{aligned}
\right.
\end{equation}

Large deviations provides asymptotic estimates for probabilities of rare events and have wide applications in fields such as physics, communication engineering, biology, finance and computer science. There are a large number of papers devoted to LDP of SEEs and SPDEs driven by Brownian motions; see \cite{BDV1,C1,CR,C2} and references therein. And one can also refer to \cite{DM} for LDP for Boussinesq equations; \cite{L} for LDP for SEEs with small multiplicative noise and \cite{MSS} for LDP for the stochastic shell model.

An effective approach to obtain the LDP is the weak convergence method introduced by Budhiraja, Dupuis and Maroulas in \cite{BD,BDV1,BDV2}. However, due to the singularity introduced by the reflection, proving the tightness of the perturbations of reflected SEEs required by the creteria in \cite{BD} is challenging. Instead, we will use a sufficient condition of the  Budhiraja–Dupuis–Maroulas criteria proposed in \cite{MSZ}, which shifts the difficulty of proving tightness to a study of the continuity (with respect to the driving signals) of the deterministic skeleton equations associated with the reflected SEEs. Therefore, a crucial part of our work is to study the continuity of the deterministic reflection problems driven by the elements in the Cameron–Martin space of the driving Brownian motions.

The paper is organized as follows. In Section 2, we introduce the precise assumptions of coefficients and the result of existence and uniqueness of solution for (\ref{SEE}). In Section 3, we introduce the definition of LDP and the sufficient condition proposed in \cite{MSZ}. In Section 4, we study the skeleton reflection problems and establish the continuity of the deterministic reflection problems with respect to the driving signal. The proof of the large deviation principle will be presented in Section 5. And in Section 6 we give a example of damped Navier-Stokes Equations in the whole Euclidean domain $\mathbb{R}^2$.

\section{Framework}
\setcounter{equation}{0}
 \setcounter{definition}{0}

Let $A$ be a self-adjoint, positive define operator on the Hilbert space $H$ such that there exists $\lambda_1>0$ satisfying
\begin{equation}
(Au,u)\geq\lambda_1|u|^2, \quad u\in D(A).
\end{equation}

Set $V:=D(A^{\frac{1}{2}})$, the domain of the operator $A^{\frac{1}{2}}$. Then V is a Hilbert space with the inner product
\begin{equation}
((u,v))=(A^{\frac{1}{2}}u,A^{\frac{1}{2}}v), \quad u,v\in V,
\end{equation}
and the norm$\|\cdot\|$. By $V^{\ast}$ we denote the dual space of $V$, so that we have a Gelfand triple
\begin{equation}
V\hookrightarrow H= H^{\ast} \hookrightarrow V^{\ast}.
\end{equation}
We also use $\langle\cdot,\cdot\rangle$ to denote the duality between $V$ and $V^{\ast}$.

The following hypothesis will be in effect throughout this work.\\
\textbf{(A.1)} Consider $f:H\rightarrow V^{\ast}$ and $\sigma:H\rightarrow H$ are two measurable maps such that there exists a constant $C$ satisfying
\begin{equation}\label{A1}
|f(u)-f(v)|_{V^{\ast}} + |\sigma(u)-\sigma(v)| \leq C|u-v|, \quad \text{for all }u,v\in H.
\end{equation}
\textbf{(A.2)} Consider a bilinear map $B:V\times V \rightarrow V^{\ast}$ and the corresponding trilinear form $\overline{b}:V\times V\times V\rightarrow R$ defined by
\begin{equation}
\overline{b}(u,v,w)=\langle B(u,v),w\rangle, \quad u,v,w\in V.
\end{equation}
And the form $\overline{b}$ satisfies the following conditions.\\
a) For all $u,v,w\in V$,
\begin{equation}
\langle B(u,v),w\rangle= \overline{b}(u,v,w)= -\overline{b}(u,w,v)= -\langle B(u,w),v\rangle.
\end{equation}
b) For all $u,v,w\in V$,
\begin{equation}
|\langle B(u,v),w\rangle| = |\overline{b}(u,v,w)|\leq 2 \|u\|^{\frac{1}{2}} |u|^{\frac{1}{2}} \|w\|^{\frac{1}{2}}|w|^{\frac{1}{2}}\|v\|.
\end{equation}
And we can observe that the constant 2 can be replaced by any positive constant.\\
\textbf{(A.3)} $u_0\in \overline{D}$ is deterministic.

Assumption (A.2) particularly implies that
\begin{equation}
\begin{aligned}
&\overline{b}(u,v,v)=0 \quad i.e.\quad\langle B(u,v),v\rangle=0, u,v\in V,\\
&\|B(u,u)\|_{V^{\ast}}\leq 2\|u\||u|, u\in V.
\end{aligned}
\end{equation}
For a given $T>0$, let $X_T=C([0,T];H)\cap L^2([0,T];V)$ be the Banach space endowed with the norm
\begin{equation}\nonumber
|u|^2_{X_T}:=\sup\limits_{t\in[0,T]}|u(t)|^2 + \int_0^T\|u(s)\|^2ds.
\end{equation}
Throughout the paper, $C$ will denote a generic constant whose value may be different from line to line.
\begin{definition}
A pair $(u,L)$ is said to be a solution of Problem $(\ref{SEE})$ iff the following conditions are satisfied\\
(i) $u$ is a $\overline{D}$-valued continuous and $\mathbb{F}$-progressively measurable stochastic process with $u\in L^2([0,T];V)$, for any $T>0$, $\mathbb{P}$-a.s.\\
(ii) the corresponding $V$-valued process is strongly $\mathbb{F}$-progressively measurable;\\
(iii) $L$ is $H$-valued, $\mathbb{F}$-progressively measurable stochastic process of paths of locally bounded  variation such that $L(0)=0$ and
\begin{equation}
\mathbb{E}[|\text{Var}_H(L)([0,T])|^2]<+ \infty,\quad T\geq 0,
\end{equation}
where, for a function $v:[0,\infty)\rightarrow H$,Var$_H(v)([0,T])$ is the total variation of $v$ on $[0,T]$ defined by
\begin{equation}
\text{Var}_H(L)([0,T]):= sup\sum_{i=1}^n|v(t_i)-v(t_{i-1})|,
\end{equation}
where the supremum is taken over all partitions $0 = t_0 < t_1 <\cdots < t_{n-1} < t_n = T$, $n\in \mathbb{N}$, of the interval $[0,T]$;\\
(iv) $(u,L)$ satisfies the following integral identity in $V^{\ast}$, for every $t\geq 0$, $\mathbb{P}$-almost surely,
\begin{equation}\nonumber
u(t)+ \int_0^t Au(s)ds -\int_0^t f(u(s))ds - \int_0^t B(u(s),u(s))ds= u_0 +\int_0^t\sigma(u(s))dW(s) +L(t);\\
\end{equation}
(v) for every $T>0$, and $\phi \in C([0,T],\overline{D})$, $\mathbb{P}$-almost surely,
\begin{equation} \label{def5}
\int_0^T (\phi(t)-u(t),L(dt))\geq 0.
\end{equation}
where the integral on LHS is Riemann-Stieltjes integral of the $H$- valued function $\phi-u$ with respect to an $H$-valued bounded-variation function $L$, see \cite{BP}(p47).
\end{definition}
Let us recall the following result from \cite{BZ}(Theorem 4.1).
\begin{theorem} Let the assumptions \textbf{(A.1)-(A.3)} hold. The reflected stochastic evolution equation(\ref{SEE}) admits a unique solution $(u,L)$ that satisfies, for $T>0$,
\begin{equation}
\mathbb{E}[\sup\limits_{t\in[0,T]}|u(t)|^2 + \int_0^T\|u(s)\|^2ds]<\infty.
\end{equation}
\end{theorem}

\section{A sufficient Condition for LDP}
\setcounter{equation}{0}
 \setcounter{definition}{0}
In this section we recall the sufficient condition proposed in \cite{MSZ} for the criteria obtained in \cite{BD}
for proving a large deviation principle.
Let $\mathcal{E}$ be a Polish space with the Borel $\sigma$-field $\mathcal{B}(\mathcal{E})$. Recall
\begin{definition}(Rate function) A function $I:\mathcal{E}\rightarrow[0,\infty]$ is called a rate function on $\mathcal{E}$, if for each $M<0$, the level set $\{x\in\mathcal{E}:I(x)\leq M\}$ is a compact subset of $\mathcal{E}$.
\end{definition}
\begin{definition}(Large deviation principle) Let $I$ be a rate function on $\mathcal{E}$. A family $\{X^{\epsilon}\}$ of $\mathcal{E}$-valued random elements is said to satisfy a large deviation principle on $\mathcal{E}$ with rate function $I$ if the following two claims hold.\\
(a) (Upper bound) For each closed subset $F$ of $\mathcal{E}$,
\begin{equation}\nonumber
\limsup\limits_{\epsilon\rightarrow 0}\epsilon \textup{log} \mathbb{P}(X^{\epsilon}\in F)\leq - \inf\limits_{x\in F} I(x).
\end{equation}
(b) (Lower bound) For each open subset $G$ of $\mathcal{E}$,
\begin{equation}\nonumber
\liminf\limits_{\epsilon \rightarrow 0}\epsilon \textup{log}\mathbb{P}(X^{\epsilon}\in G)\geq - \inf\limits_{x\in F} I(x).
\end{equation}
\end{definition}
 The Cameron-Martin space associated with the Browian motion $\{W=W(t),t\in [0,T]\}$ is isomorphic to the Hilbert space $K:=L^2([0,T];\mathbb{R})$ with the inner product:
\begin{equation}
\langle h_1,h_2\rangle  _K:=\int_0^T h_1(s)h_2(s) ds.
\end{equation}
Denote by $\|\cdot\|_K$ the norm of K.
Let $\widetilde{K}$ denote the class of real-valued $\{\mathcal{F}_t\}$-predictable processes $\phi$ that belong to the space $K$ a.s.. Let $S_N=\{k\in K;\int_0^T k^2(s)ds\leq N\}.$ The set $S_N$ endowed with the weak topology is a compact Polish space. Set $\widetilde{S}_N=\{\phi\in \widetilde{K}; \phi(\omega)\in S_N,\mathbb{P}$-$a.s.\}.$\\
Recall the following sufficient condition from \cite{MSZ}.

\begin{theorem} \label{cri} For $\varepsilon>0$, let $\Gamma^{\varepsilon}$ be a measurable mapping from $C([0,T],\mathbb{R})$ into $\mathcal{E}$. Set $X^{\varepsilon}:= \Gamma^{\varepsilon}(W(\cdot))$. Suppose that there exists a measurable map $\Gamma^0:C([0,T],\mathbb{R})\rightarrow \mathcal{E}$ such that\\
(i) for every $N<+\infty$, any family $\{k^{\varepsilon};\varepsilon >0\}\subset \widetilde{S}_N$ and  any $\delta>0$
\begin{equation}
\lim\limits_{\varepsilon \rightarrow 0} \mathbb{P}(\rho(Y^{\varepsilon},Z^{\varepsilon})>\delta)=0,
\end{equation}
where $Y^{\varepsilon}=\Gamma^{\varepsilon}(W(\cdot)+\frac{1}{\sqrt{\varepsilon}}\int _0^\cdot k^{\varepsilon}(s)ds), Z^{\varepsilon}=\Gamma^0(\int _0^\cdot k^{\varepsilon}(s)ds)$ and $\rho(\cdot,\cdot)$ stands for the metric in the space $\mathcal{E}$.\\
(ii) for every $N<+\infty$ and any family $\{k^{\varepsilon};\varepsilon >0\}\subset S_N$ satisfying that $k^{\varepsilon}$ converges weakly to some element $k$ in $K$ as $\varepsilon \rightarrow 0,$ $ \Gamma^0(\int _0^\cdot k^{\varepsilon}(s)ds)$ converges to $\Gamma^0(\int _0^\cdot k(s)ds)$ in the space $\mathcal{E}$.\\
Then the family $\{X^{\varepsilon}\}_{\varepsilon >0}$ satisfies a large deviation principle in $\mathcal{E}$ with the rate function $I$ given by
\begin{equation}
I(g):=\inf\limits_{\{k\in K, g=\Gamma^0(\int _0^\cdot k(s)ds)\} } \left\{\frac{1}{2}\int _0^T k^2(s)ds\right\},g\in \mathcal{E},
\end{equation}
with the convention $\inf\{\emptyset\}=\infty$.
\end{theorem}          

\section{Skeleton Equations}
\setcounter{equation}{0}
 \setcounter{definition}{0}
For $k\in K:=L^2([0,T];\mathbb{R})$, consider the deterministic reflection problem:
\begin{equation}\label{deter} 
\left\{
\begin{aligned}
du^k(t)&+Au^k(t)dt =f(u^k(t))dt +B(u^k(t),u^k(t))dt + \sigma(u^k(t))k(t)dt+dL^k(t), t\geq0,\\
 u^k(0)& =u_0, \quad u_0\in \overline{D}.
\end{aligned}
\right.
\end{equation}
The existence and uniqueness of the solution of the deterministic reflection problem (\ref{deter}) can be obtained in a same way  as the stochastic reflection problem (\ref{SEE}). Denote by $u^{k^\varepsilon}$ the solution of equation(\ref{deter}) with $k^\varepsilon$ replacing $k$. The main purpose of this section is to show that $u^{k^\varepsilon}$ converges to $u^k$ in the space $X_T$ if $k^\varepsilon\rightarrow k$ weakly in the Hilbert space $K$. To this end, we first need to establish a number of preliminary results.
Consider the associated penalized equation:
\begin{equation}\label{penal} 
\left\{
\begin{aligned}
du^{k,n}(t)+Au^{k,n}(t)dt=&f(u^{k,n}(t))dt +B(u^{k,n}(t),u^{k,n}(t))dt \\
                         & \quad + \sigma(u^{k,n}(t))k(t)dt-n(u^{k,n}(t)-\pi(u^{k,n}(t)))dt, t\geq0,\\
  u^{k,n}(0) =u_0, \quad u_0\in \overline{D},
\end{aligned}
\right.
\end{equation}
where $\pi$ is the projection onto $\overline{D}$, defined by, for $y\in H$, by
\begin{equation}
\pi(y)=\begin{cases} y, \quad  &\text{if }|y|\leq 1,\\
              \frac{y}{|y|}, \quad &\text{if } |y|>1.
\end{cases}
\end{equation}
As in the case of reflected SEEs in \cite{BZ}(but simpler), one can show that $u^{k,n}\rightarrow u^k$ as $n\rightarrow \infty$ for a fixed $k \in K$. We need to show that  for any $N>0, u^{k,n}\rightarrow u^k$ uniformly over the bounded subset $S_N$ as $n\rightarrow \infty$.
First, similarly as in \cite{BZ} we have the following estimates.
\begin{proposition}\label{prop}Let the assumptions \textbf{(A.1)-(A.3)} hold.\\
1. For $x\in H$ and $y \in \overline{D},$
\begin{equation}\label{1} 
(x-y,x-\pi(x))\geq 0.
\end{equation}
and the map $\pi$ is globally Lipschitz in the sense that
\begin{equation}\label{4}  
|\pi(x)-\pi(y)|\leq 2 |x-y|,\quad x,y \in H.
\end{equation}
2. For fixed $k\in K$ and $T>0$ there exists a constant $C_{T}$ such that
\begin{equation}\label{6} 
\sup\limits_n\sup \limits_{t\in [0,T]}|u^{k,n}(t)|^4\leq C_{T}.
\end{equation}
3. For $N>0$ and $T>0$,  there exists a constant $C_{N,T}$ such that
\begin{equation}\label{2} 
\sup\limits_{k\in S_N}\sup \limits_n \left(n\int_0^T|u^{k,n}(s)-\pi(u^{k,n}(s))|ds\right)\leq C_{N,T},
\end{equation}
\begin{equation} \label{3} 
\lim \limits_{n\rightarrow \infty} \sup\limits_{k\in S_N}\sup \limits_{t\in [0,T]}|u^{k,n}(s)-\pi(u^{k,n}(s))|^4=0,
\end{equation}
and
\begin{equation}\label{5} 
\sup\limits_{k\in S_N}\sup \limits_n \left(n\int_0^T|u^{k,n}(s)-\pi(u^{k,n}(s))|^2ds \right)\leq C_{N,T}.
\end{equation}

\end{proposition}
\begin{proof}
One can refer to Lemma 3.2 in \cite{BZ} for the proof of the statement $1$. (\ref{6}) and (\ref{3}) can be obtained following the similar arguments as in the proof of \cite{BZ} Lemma 3.4 and Lemma 3.8 in \cite{BZ}.  Using the chain rule, as in the proof of Lemma 3.6  we have
\begin{equation}\nonumber
\begin{aligned}
\sup \limits_{n}&\sup \limits_{t\in [0,T]}|u^{k,n}(t)|^2  + \sup \limits_{n}\int_0^T\|u^{k,n}(s)\|^2ds\\
&\hspace{4cm} +\sup \limits_{n}\left(2n\int_0^T|u^{k,n}(s)-\pi(u^{k,n}(s))|^2ds + 2 n\int_0^T|u^{k,n}(s)-\pi(u^{k,n}(s))|ds\right)\\
&\leq C + C\int_0^T|k(s)||u^{k,n}(s)|^2ds.
\end{aligned}
\end{equation}
Then it follows from Gronwall inequality that there exists a constant $C_{N,T}$ such that
\begin{equation}\nonumber
\begin{aligned}
&\sup\limits_{k\in S_N}\sup \limits_{n}\int_0^T\|u^{k,n}(s)\|^2ds \leq C_{N,T},\\
&\sup\limits_{k\in S_N}\sup \limits_n \left(n\int_0^T|u^{k,n}(s)-\pi(u^{k,n}(s))|ds\right)\leq C_{N,T},\\
&\sup\limits_{k\in S_N}\sup \limits_n \left(n\int_0^T|u^{k,n}(s)-\pi(u^{k,n}(s))|^2ds \right)\leq C_{N,T}.
\end{aligned}
\end{equation}
The proof is complete.
\end{proof}

\begin{theorem}\label{4.1} Let assumptions \textbf{(A.1)-(A.3)} hold. Assume that $\{k^\varepsilon;\varepsilon>0\}\subset S_N$ satisfying that $k^\varepsilon\rightarrow k$ weakly in the Hilbert space $K$ as $\varepsilon\rightarrow 0$. Then $u^{k^\varepsilon}$ converges to $u^k$ in the space $X_T$, where $u^{k^\varepsilon}$ denotes the solution of equation $(\ref{deter})$ with $k^\varepsilon$ replacing $k$.
\end{theorem} 
\begin{proof}
For any $n\geq 1$, we have
\begin{equation}\label{10}
\begin{aligned}
&|u^{k^\varepsilon}-u^k|_{X_T}\\
&\leq|u^{k^\varepsilon}-u^{k^\varepsilon,n}|_{X_T} +|u^{k^\varepsilon,n}-u^{k,n}|_{X_T}+|u^{k,n}-u^{k}|_{X_T}.
\end{aligned}
\end{equation}
The proof will be complete if we prove

(i) for every $n\geq 1$,
\begin{equation}\label{10-1}
\lim\limits_{\varepsilon\rightarrow 0}|u^{k^\varepsilon,n}-u^{k,n}|_{X_T}=0,
\end{equation}
and

(ii)
\begin{equation}\label{10-2}
\lim\limits_{n\rightarrow \infty}\sup\limits_\varepsilon|u^{k^\varepsilon}-u^{k^\varepsilon,n}|_{X_T}=0.
\end{equation}

These two claims (\ref{10-1}) and (\ref{10-2}) will be shown in Proposition \ref{prop2} and Proposition \ref{prop1} below.

\end{proof}

\begin{proposition} \label{prop1}With the assumptions in Theorem \ref{4.1} hold. For any $N>0$, we have
\begin{equation}
\lim\limits_{n\rightarrow\infty}\sup\limits_{k\in S_N}|u^{k,n}-u^k|_{X^T}=0.
\end{equation}
\end{proposition}
\begin{proof} Recall that for each fixed $k \in K$, $u^{k,n}\rightarrow u^k$ as $n\rightarrow \infty$. Thus, it's sufficient to show
\begin{equation}
\lim\limits_{n,m\rightarrow\infty}\sup\limits_{k\in S_N}\left\{\sup\limits_{0\leq t\leq T}|u^{k,n}(t) -u^{k,m}(t)|^2+\int_0^T\|u^{k,n}(r) -u^{k,m}(r)\|^2dr\right\}=0.
\end{equation}
For $\lambda>4$ and $n\in\mathbb{N}$, define the function $f_n(t)=exp(-\lambda\int_0^t\|u^{k,n}(s)\|^2ds),$ $t\geq 0$.\\
Let $m$, $n$ be integers with $m\geq n$. Using the chain rule we can infer that
\begin{equation}
\begin{aligned}
f_n(t)|u^{k,n}(t) -u^{k,m}(t)|^2\\
=& -\lambda\int_0^t f_n(s)\|u^{k,n}(s)\|^2|u^{k,n}(s) -u^{k,m}(s)|^2 ds\\
 & -2\int_0^t f_n(s)(u^{k,n}(s) -u^{k,m}(s), A(u^{k,n}(s) -u^{k,m}(s)))ds\\
 & +2\int_0^t f_n(s)\langle u^{k,n}(s) -u^{k,m}(s), \sigma(u^{k,n}(s)) -\sigma(u^{k,m}(s))\rangle k(s)ds\\
 & +2\int_0^t f_n(s)\langle u^{k,n}(s) -u^{k,m}(s), f(u^{k,n}(s)) -f(u^{k,m}(s))\rangle ds\\
 & +2\int_0^t f_n(s)\langle u^{k,n}(s) -u^{k,m}(s), B(u^{k,n}(s),u^{k,n}(s)) -B(u^{k,m}(s),u^{k,m}(s))\rangle ds\\
 & -2n\int_0^t f_n(s)\langle u^{k,n}(s) -u^{k,m}(s), u^{k,n}(s) -\pi(u^{k,n}(s))\rangle ds\\
 & +2m\int_0^t f_n(s)\langle u^{k,n}(s) -u^{k,m}(s), u^{k,m}(s) -\pi(u^{k,m}(s))\rangle  ds\\
=&\ I_1^{n,m}+I_2^{n,m}+I_3^{n,m} +I_4^{n,m}+I_5^{n,m}+I_6^{n,m} +I_7^{n,m}.\\
\end{aligned}
\end{equation}
Observe that
\begin{equation}
I_2^{n,m}=-2\int_0^t f_n(s)\|u^{k,n}(s) -u^{k,m}(s)\|^2 ds.
\end{equation}
By the assumptions on $f$, $\sigma$, $H\ddot{o}lder's$ inequality and Young's inequality, we obtain
\begin{equation}
\begin{aligned}
I_3^{n,m}&\leq C\int_0^t f_n(s)| u^{k,n}(s) -u^{k,m}(s)| | \sigma(u^{k,n}(s)) -\sigma(u^{k,m}(s))| |k(s)| ds\\
         &\leq \frac{1}{2}\sup\limits_{0\leq s \leq t} f_n(s)| u^{k,n}(s) -u^{k,m}(s)|^2  + C \int_0^t f_n(s)| u^{k,n}(s) -u^{k,m}(s)|^2 |k(s)|^2ds.\\
I_4^{n,m}&\leq C\int_0^t f_n(s)\| u^{k,n}(s) -u^{k,m}(s)\| \| f(u^{k,n}(s)) -f(u^{k,m}(s))\|_{V^\ast} ds\\
         &\leq C\int_0^t f_n(s)\| u^{k,n}(s) -u^{k,m}(s)\| | u^{k,n}(s) -u^{k,m}(s)| ds\\
         &\leq \frac{1}{2}\int_0^t f_n(s)\| u^{k,n}(s) -u^{k,m}(s)\|^2 ds + C \int_0^t f_n(s)| u^{k,n}(s) -u^{k,m}(s)|^2 ds.
\end{aligned}
\end{equation}
Recalling the properties of  $\overline{b}$ and $B$, we have
\begin{equation}
\overline{b}(u^{k,m}(s),u^{k,m}(s),u^{k,n}(s)-u^{k,m}(s))=\overline{b}(u^{k,m}(s),u^{k,n}(s),u^{k,n}(s)-u^{k,m}(s)).\\
\end{equation}
Then, by \textbf{(A.2)}, \textbf{(A.3)}
\begin{equation}
\begin{aligned}
I_5^{n,m}&= 2\int_0^t f_n(s)\big(\overline{b}(u^{k,n}(s), u^{k,n}(s),u^{k,n}(s) -u^{k,m}(s))-\overline{b}(u^{k,m}(s),u^{k,n}(s), u^{k,n}(s) -u^{k,m}(s))\big) ds\\
&\leq 2\int_0^t f_n(s)\big|\overline{b}(u^{k,n}(s), u^{k,n}(s),u^{k,n}(s) -u^{k,m}(s))-\overline{b}(u^{k,m}(s),u^{k,n}(s), u^{k,n}(s) -u^{k,m}(s))\big|ds\\
&\leq 2\int_0^t f_n(s)\big|\overline{b}(u^{k,n}(s)-u^{k,m}(s), u^{k,n}(s),u^{k,n}(s) -u^{k,m}(s))|ds\\
&\leq 4\int_0^t f_n(s)|u^{k,n}(s)-u^{k,m}(s)|\|u^{k,n}(s)-u^{k,m}(s)\|  \|u^{k,n}(s)\|ds\\
&\leq \int_0^t f_n(s)\|u^{k,n}(s)-u^{k,m}(s)\|^2 ds +\int_0^t|u^{k,n}(s)-u^{k,m}(s)|^2 \|u^{k,n}(s)\|^2ds,\\
\end{aligned}
\end{equation}
where the inequality $4ab\leq a^2 +4b^2$ has been used in the last inequality.\\
It follows from (\ref{1}) that $\langle u^{k,n}(s) -\pi(u^{k,m}(s)), u^{k,n}(s) -\pi(u^{k,n}(s))\rangle\geq 0$, hence
\begin{equation}
\begin{aligned}
I_6^{n,m}= &-2n\int_0^t f_n(s)\langle u^{k,n}(s) -u^{k,m}(s), u^{k,n}(s) -\pi(u^{k,n}(s))\rangle ds\\
= &-2n\int_0^t f_n(s)\langle u^{k,n}(s) -\pi(u^{k,m}(s)), u^{k,n}(s) -\pi(u^{k,n}(s))\rangle ds\\
 &+2n\int_0^t f_n(s)\langle u^{k,m}(s) -\pi(u^{k,m}(s)), u^{k,n}(s) -\pi(u^{k,n}(s))\rangle ds\\
\leq &\quad 2n\int_0^t f_n(s)\langle u^{k,m}(s) -\pi(u^{k,m}(s)), u^{k,n}(s) -\pi(u^{k,n}(s))\rangle ds\\
\leq &\quad 2n(\int_0^t |u^{k,n}(s) -\pi(u^{k,n}(s))|ds)\sup\limits_{0\leq s \leq t} |u^{k,m}(s) -\pi(u^{k,m}(s))|.
\end{aligned}
\end{equation}
And similarly
\begin{equation}
I_7^{n,m}\leq \ 2m(\int_0^t |u^{k,m}(s) -\pi(u^{k,m}(s))|ds)\sup\limits_{0\leq s \leq t} |u^{k,n}(s) -\pi(u^{k,n}(s))|.
\end{equation}
Combining the above estimates together, remembering that  $\lambda>4$, using the H$\ddot{o}$lder inequalities, we obtain that
\begin{equation}
\begin{aligned}
&\sup\limits_{0\leq t \leq T}f_n(t)|u^{k,n}(t) -u^{k,m}(t)|^2 + \int_0^T f_n(s)\|u^{k,n}(s) -u^{k,m}(s)\|^2ds\\
\leq \ &\frac{1}{2}\sup\limits_{0\leq t \leq T}f_n(t)|u^{k,n}(t) -u^{k,m}(t)|^2+ C \int_0^T f_n(s)(1+|k(s)|^2)|u^{k,n}(s) -u^{k,m}(s)|^2ds\\
 + &\ C\left(2n\int_0^T |u^{k,n}(s) -\pi(u^{k,n}(s))|ds\right)\sup\limits_{0\leq t \leq T}|u^{k,m}(t) -\pi(u^{k,m}(t))|\\
 +&\ C\left(2m\int_0^T |u^{k,m}(s) -\pi(u^{k,m}(s))|ds\right) \sup\limits_{0\leq t \leq T} |u^{k,n}(t) -\pi(u^{k,n}(t))|.
\end{aligned}
\end{equation}
In view of (\ref{2}), by Gronwall Lemma we get
\begin{equation}\label{10-3}
\begin{aligned}
&\sup\limits_{0\leq t \leq T}f_n(t)|u^{k,n}(t) -u^{k,m}(t)|^2 + \int_0^T f_n(s)\|u^{k,n}(s) -u^{k,m}(s)\|^2ds\\
&\leq \ C(C_{N,T})\sup\limits_{0\leq t \leq T}|u^{k,m}(t) -\pi(u^{k,m}(t))| +C(C_{N,T})\sup\limits_{0\leq t \leq T} |u^{k,n}(t) -\pi(u^{k,n}(t))|.
\end{aligned}
\end{equation}
Then by (\ref{3}), it follows that
\begin{equation}
\lim\limits_{n,m\rightarrow\infty}\sup\limits_{k\in S_N}\left\{\sup\limits_{0\leq t\leq T}|u^{k,n}(t) -u^{k,m}(t)|^2+\int_0^T\|u^{k,n}(r) -u^{k,m}(r)\|^2dr\right\}=0.
\end{equation}
\end{proof}

\begin{proposition}\label{prop2} With the assumptions in Theorem \ref{4.1} hold. For fixed integer $n\in \mathbb{N}$,
\begin{equation}
\lim\limits_{\varepsilon\rightarrow 0}|u^{k^\varepsilon,n}-u^{k,n}|_{X_T}=0,
\end{equation}
where $u^{k^\varepsilon,n}$ denotes the solution of equation $(\ref{penal})$ with $k^\varepsilon$ replacing $k$.
\end{proposition}
Before we proceed to the proof, we recall the following lemma from \cite{NTT}(see \cite{NTT} Lemma 4.3 for more details).
\begin{lemma}\label{lem}Let $\alpha \in (0,1)$ be given. Let $\mathcal{G}$ be the space
\begin{equation}
\mathcal{G}=L^\infty(0,T;H)\cap L^2(0,T;V) \cap W^{\alpha,2}(0,T;V^{\ast})
\end{equation}
endowed with the natural norm. Then the embedding of $\mathcal{G}$ in $L^2(0,T;H)$ is compact.
\end{lemma}           

\hspace{-1.5em}\textit{{\bf Proof of Proposition 4.3}}
\vspace{1em}\\
First we prove that the family $\{u^{k^\varepsilon,n},\varepsilon > 0\}$ is relatively compact in $L^2(0,T;H)$. To this end, by Lemma \ref{lem}, it is sufficient to show that $\{u^{k^\varepsilon,n},\varepsilon > 0\}$ is a bounded subset of the space $\mathcal{G}$.  As $\{k^\varepsilon;\varepsilon>0\}\subset S_N$, we infer from the proof of Proposition \ref{prop} that $\{u^{k^\varepsilon,n},\varepsilon > 0\}$ is a bounded subset of the space $L^\infty(0,T;H)\cap L^2(0,T;V)$. Next we show that $\{u^{k^\varepsilon,n},\varepsilon > 0\}$ is also bounded in the space $W^{\alpha,2}(0,T;V^{\ast})$. From the equation satisfied by $u^{k^\varepsilon,n}$ we have
\begin{equation}\label{10-4}
\begin{aligned}
&\ \|u^{k^\varepsilon,n}\|_{W^{\alpha,2}(0,T;V^{\ast})}\leq \|u^{k^\varepsilon,n}\|_{W^{1,2}(0,T;V^{\ast})}\\
\lesssim &\ \|u_0 -\int_0^\cdot Au^{k^\varepsilon,n}(s)ds + \int_0^\cdot f(u^{k^\varepsilon,n}(s))ds + \int_0^\cdot B(u^{k^\varepsilon,n}(s),u^{k^\varepsilon,n}(s))ds \\
&\hspace{5cm}+\int_0^\cdot \sigma(u^{k^\varepsilon,n}(s))k^\varepsilon(s)ds - n\int_0^\cdot(u^{k^\varepsilon,n}(s)-\pi(u^{k^\varepsilon,n}(s))ds\|^2_{{W^{1,2}(0,T;V^{\ast})}}\\
\lesssim &\ \|u_0 -\int_0^\cdot Au^{k^\varepsilon,n}(s)ds + \int_0^\cdot f(u^{k^\varepsilon,n}(s))ds + \int_0^\cdot B(u^{k^\varepsilon,n}(s),u^{k^\varepsilon,n}(s))ds\\
&\hspace{5cm} +\int_0^\cdot \sigma(u^{k^\varepsilon,n}(s))k^\varepsilon(s)ds - n\int_0^\cdot(u^{k^\varepsilon,n}(s)-\pi(u^{k^\varepsilon,n}(s))ds\|^2_{{L^2(0,T;V^{\ast})}}\\
&\ +\int_0^T |Au^{k^\varepsilon,n}(s)|^2_{V^{\ast}}ds + \int_0^T|f(u^{k^\varepsilon,n}(s))|^2_{V^{\ast}}ds + \int_0^T| B(u^{k^\varepsilon,n}(s),u^{k^\varepsilon,n}(s))|^2_{V^{\ast}}ds\\
&\hspace{5cm} +\int_0^T |\sigma(u^{k^\varepsilon,n}(s))k^\varepsilon(s)|^2_{V^{\ast}}ds + \int_0^T|n(u^{k^\varepsilon,n}(s)-\pi(u^{k^\varepsilon,n}(s)))|^2_{V^{\ast}}ds\\
\lesssim&\ |u_0|^2 + \int_0^T |Au^{k^\varepsilon,n}(s)|^2_{V^{\ast}}ds + \int_0^T|f(u^{k^\varepsilon,n}(s))|^2_{V^{\ast}}ds + \int_0^T| B(u^{k^\varepsilon,n}(s),u^{k^\varepsilon,n}(s))|^2_{V^{\ast}}ds\\
&\hspace{5cm} +\int_0^T |\sigma(u^{k^\varepsilon,n}(s))k^\varepsilon(s)|^2_{V^{\ast}}ds + \int_0^T|n(u^{k^\varepsilon,n}(s)-\pi(u^{k^\varepsilon,n}(s)))|^2_{V^{\ast}}ds\\
= &\ |u_0|^2 + K_1 + K_2 + K_3 + K_4 + K_5.
\end{aligned}
\end{equation}
Recalling the assumptions on the coefficients, we have
\begin{equation}\label{10-5}
\begin{aligned}
K_1&\leq\int_0^T\|u^{k^\varepsilon,n}(s)\|^2ds.\\
K_2&\leq\int_0^T C(1+ |u^{k^\varepsilon,n}(s)|^2)ds.\\
K_3&\leq\int_0^T2\|u^{k^\varepsilon,n}(s)\|^2|u^{k^\varepsilon,n}(s)|^2ds\leq 2 \sup \limits_{0\leq t \leq T}|u^{k^\varepsilon,n}(t)|^2 \int _0^T\|u^{k^\varepsilon,n}(s)\|^2ds.\\
K_4&\leq \int_0^T| C \sigma(u^{k^\varepsilon,n}(s))k^\varepsilon(s)|^2ds\leq C \sup \limits_{0\leq t \leq T}(1+|u^{k^\varepsilon,n}(t)|^2 )\int_0^T|k^\varepsilon(s)|^2 ds.
\end{aligned}
\end{equation}
and by (\ref{5}),
\begin{equation}
K_5\leq \int_0^T|n(u^{k^\varepsilon,n}(s)-\pi(u^{k^\varepsilon,n}(s)))|^2ds\leq n(\sup \limits_n n\int_0^T|u^{k,n}(s)-\pi(u^{k,n}(s))|^2ds)\leq n C_{N,T}.
\end{equation}
Since $\{u^{k^\varepsilon,n},\varepsilon > 0\}$ is bounded in $\subset L^\infty(0,T;H)\cap L^2(0,T;V)$, we deduce from (\ref{10-4}), (\ref{10-5}) that $\sup\limits_\varepsilon \|u^{k^\varepsilon,n}\|_{W^{\alpha,2}(0,T;V^{\ast})}<\infty$.
Hence, by Lemma \ref{lem}, $\{u^{k^\varepsilon,n}, \varepsilon>0\}$ is relatively compact in $L^2(0,T;H)$. Now we are ready to show that $u^{k^\varepsilon,n}\rightarrow u^{k,n}$ in the space $X_T$. \\
For $\lambda>4$, define a function $f_n(t)=exp(-\lambda\int_0^t\|u^{k,n}(s)\|^2ds),$ $t\geq 0$. By the chain rule, we have
\begin{equation}\label{j}
\begin{aligned}
f_n(t)|u^{k^\varepsilon,n}(t) -u^{k,n}(t)|^2\\
=& -\lambda\int_0^t f_n(s)\|u^{k,n}(s)\|^2|u^{k^\varepsilon,n}(s) -u^{k,n}(s)|^2 ds\\
 & -2\int_0^t f_n(s)(u^{k^\varepsilon,n}(s) -u^{k,n}(s), A(u^{k^\varepsilon,n}(s) -u^{k,n}(s)))ds\\
 & +2\int_0^t f_n(s)\langle u^{k^\varepsilon,n}(s) -u^{k,n}(s), \sigma(u^{k^\varepsilon,n}(s)) -\sigma(u^{k,n}(s))\rangle k^\varepsilon(s)ds\\
 & +2\int_0^t f_n(s)\langle u^{k^\varepsilon,n}(s) -u^{k,n}(s), \sigma(u^{k,n}(s))\rangle( k^\varepsilon(s)-k(s))ds\\
 & +2\int_0^t f_n(s)\langle u^{k^\varepsilon,n}(s) -u^{k,n}(s), f(u^{k^\varepsilon,n}(s)) -f(u^{k,n}(s))\rangle ds\\
 & +2\int_0^t f_n(s)\langle u^{k^\varepsilon,n}(s) -u^{k,n}(s), B(u^{k^\varepsilon,n}(s),u^{k^\varepsilon,n}(s)) -B(u^{k,n}(s),u^{k,n}(s))\rangle ds\\
 & -2n\int_0^t f_n(s)\langle u^{k^\varepsilon,n}(s) -u^{k,n}(s), u^{k^\varepsilon,n}(s) -\pi(u^{k^\varepsilon,n}(s))\rangle ds\\
 & +2n\int_0^t f_n(s)\langle u^{k^\varepsilon,n}(s) -u^{k,n}(s), u^{k,n}(s) -\pi(u^{k,n}(s))\rangle  ds\\
=&\ J_1^{n}+J_2^{n}+J_3^{n} +J_4^{n}+J_5^{n}+J_6^{n} +J_7^{n}+ J_8^{n}.\\
\end{aligned}
\end{equation}
Similarly  as in the proof of Proposition \ref{prop1}, we have
\begin{equation}
\begin{aligned}
J_2^{n}&=-2\int_0^tf_n(s)\|u^{k^\varepsilon,n}(s) -u^{k,n}(s)\|^2ds,\\
J_3^{n}&\leq2\int_0^tf_n(s)|k^\varepsilon(s)||u^{k^\varepsilon,n}(s) -u^{k,n}(s)|^2ds,\\
J_5^{n}&\leq\frac{1}{2}\int_0^tf_n(s)\|u^{k^\varepsilon,n}(s) -u^{k,n}(s)\|^2ds + C\int_0^tf_n(s)|u^{k^\varepsilon,n}(s) -u^{k,n}(s)|^2ds.
\end{aligned}
\end{equation}
By the assumption on $\overline{b}$ and $B$,
\begin{equation}
\langle u^{k^\varepsilon,n}(s) -u^{k,n}(s), B(u^{k^\varepsilon,n}(s),u^{k^\varepsilon,n}(s))\rangle=\langle u^{k^\varepsilon,n}(s) -u^{k,n}(s), B(u^{k^\varepsilon,n}(s),u^{k,n}(s))\rangle.
\end{equation}
Then,
\begin{equation}
\begin{aligned}
J_6^{n}&=2\int_0^t f_n(s)\langle u^{k^\varepsilon,n}(s) -u^{k,n}(s), B(u^{k^\varepsilon,n}(s),u^{k,n}(s)) -B(u^{k,n}(s),u^{k,n}(s))\rangle ds\\
 &=2\int_0^tf_n(s)\overline{b}(u^{k^\varepsilon,n}(s) -u^{k,n}(s),u^{k,n}(s),u^{k^\varepsilon,n}(s) -u^{k,n}(s))ds\\
 &\leq4\int_0^tf_n(s)\|u^{k,n}(s)\|\|u^{k^\varepsilon,n}(s) -u^{k,n}(s)\||u^{k^\varepsilon,n}(s) -u^{k,n}(s)|ds\\
 &\leq\int_0^tf_n(s)\|u^{k^\varepsilon,n}(s) -u^{k,n}(s)\|^2ds + 4\int_0^tf_n(s)\|u^{k,n}(s)\|^2|u^{k^\varepsilon,n}(s) -u^{k,n}(s)|^2ds.
\end{aligned}
\end{equation}
By (\ref{4}),
\begin{equation}
\begin{aligned}
J_7^{n}+J_8^{n}&=2n\int_0^t f_n(s)\langle u^{k^\varepsilon,n}(s) -u^{k,n}(s), u^{k,n}- u^{k^\varepsilon,n}(s)+\pi( u^{k^\varepsilon,n}(s))-\pi(u^{k,n})\rangle ds\\
 &\leq6n\int_0^t f_n(s) |u^{k^\varepsilon,n}(s) -u^{k,n}(s)|^2ds.
\end{aligned}
\end{equation}
Combining above estimates, we get
\begin{equation}\label{10-6}
\begin{aligned}
    &f_n(t)|u^{k^\varepsilon,n}(t) -u^{k,n}(t)|^2+\frac{1}{2}\int_0^tf_n(s)\|u^{k^\varepsilon,n}(s) -u^{k,n}(s)\|^2ds\\
&\leq\ \int_0^t C n (|k^\varepsilon(s)|+1)f_n(s)|u^{k^\varepsilon,n}(s) -u^{k,n}(s)|^2ds\\
& \ \ +2\int_0^t f_n(s)\langle u^{k^\varepsilon,n}(s) -u^{k,n}(s), \sigma(u^{k,n}(s))\rangle( k^\varepsilon(s)-k(s))ds.
\end{aligned}
\end{equation}
As we know that $\{k^\varepsilon;\varepsilon>0\}\subset S_N$, to show $\lim\limits_{\varepsilon\rightarrow 0}|u^{k^\varepsilon,n}-u^{k,n}|_{X_T}=0$, by (\ref{10-6}) and the Gronwall's inequality it suffices to prove
\begin{equation}
\lim\limits_{\varepsilon\rightarrow 0}\sup\limits_{0\leq t \leq T}|\int_0^t f_n(s)\langle u^{k^\varepsilon,n}(s) -u^{k,n}(s), \sigma(u^{k,n}(s))\rangle( k^\varepsilon(s)-k(s))ds|=0.
\end{equation}
This will be achieved if we show that for any sequence $\varepsilon_m \rightarrow 0$, we can find a subsequence $\varepsilon_{mi}\rightarrow 0$ s.t.
\begin{equation}
\lim\limits_{i\rightarrow \infty}\sup\limits_{0\leq t\leq T}|\int_0^t f_n(s)\langle u^{k^{\varepsilon_{mi}},n}(s) -u^{k,n}(s), \sigma(u^{k,n}(s))\rangle( k^{\varepsilon_{mi}}(s)-k(s))ds|=0.
\end{equation}
For a fixed sequence  $\varepsilon_m \rightarrow 0$, since $\{u^{k^{\varepsilon_m},n},m\geq 1\}$ is relatively compact in $L^2(0,T;H)$, there exists a subsequence $\{u^{k^{\varepsilon_{mi}},n},i\geq 1\}$ and a mapping $\tilde{u}\in L^2(0,T;H)$ such that $u^{k^{\varepsilon_{mi}},n}\rightarrow \tilde{u}$ in $L^2(0,T;H)$. We have
\begin{equation}
\begin{aligned}
& \sup\limits_{0\leq t\leq T}|\int_0^t f_n(s)\langle u^{k^{\varepsilon_{mi}},n}(s) -u^{k,n}(s), \sigma(u^{k,n}(s))\rangle( k^{\varepsilon_{mi}}(s)-k(s))ds|\\
&\leq\ \sup\limits_{0\leq t\leq T}|\int_0^t f_n(s)\langle u^{k^{\varepsilon_{mi}},n}(s) -\tilde{u}(s), \sigma(u^{k,n}(s))\rangle( k^{\varepsilon_{mi}}(s)-k(s))ds|\\
&\ \ +\sup\limits_{0\leq t\leq T}|\int_0^t f_n(s)\langle\tilde{u}(s) -u^{k,n}(s), \sigma(u^{k,n}(s))\rangle( k^{\varepsilon_{mi}}(s)-k(s))ds|.\\
\end{aligned}
\end{equation}
Recalling that  $|(\sigma(u^{k,n}(s))|\leq C(1+|u^{k,n}(s)|)$ and  (\ref{6}), we have
\begin{equation}
\begin{aligned}
&\sup\limits_{0\leq t\leq T}|\int_0^t f_n(s)\langle u^{k^{\varepsilon_{mi}},n}(s) -\tilde{u}(s), \sigma(u^{k,n}(s))\rangle( k^{\varepsilon_{mi}}(s)-k(s))ds|\\
&\leq C \left(\int_0^T|u^{k^{\varepsilon_{mi}},n}(s) -\tilde{u}(s)|^2ds\right)^\frac{1}{2}\left(2 \int_0^T|k^{\varepsilon_{mi}}(s)|^2+|k(s)|^2 ds\right)^\frac{1}{2}\sup\limits_{0\leq t\leq T}(1+|u^{k,n}(s)|)\\
&\lesssim \left(\int_0^T|u^{k^{\varepsilon_{mi}},n}(s) -\tilde{u}(s)|^2ds\right)^\frac{1}{2}.
\end{aligned}
\end{equation}
As $u^{k^{\varepsilon_{mi}},n}\rightarrow \tilde{u}$ in $L^2(0,T;H)$, we deduce that
\begin{equation}\label{l1}
\lim\limits_{i\rightarrow \infty}\sup\limits_{0\leq t\leq T}|\int_0^t f_n(s)\langle u^{k^{\varepsilon_{mi}},n}(s) -\tilde{u}(s), \sigma(u^{k,n}(s))\rangle( k^{\varepsilon_{mi}}(s)-k(s))ds|=0.
\end{equation}
Recall that $k^{\varepsilon_{mi}}\rightarrow k$ weakly in $L^2(0,T;\mathbb{R})$, then for every $t>0$,
\begin{equation}\label{u1}
\lim\limits_{i\rightarrow \infty}\int_0^t f_n(s)\langle\tilde{u}(s) -u^{k,n}(s), \sigma(u^{k,n}(s))\rangle( k^{\varepsilon_{mi}}(s)-k(s))ds=0.
\end{equation}
On the other hand, for $0 < t_1<t_2\leq T$, we have
\begin{equation}\label{u2}
\begin{aligned}
&|\int_{t_1}^{t_2} f_n(s)\langle\tilde{u}(s) -u^{k,n}(s), \sigma(u^{k,n}(s))\rangle( k^{\varepsilon_{mi}}(s)-k(s))ds|\\
&\leq \left(\int_{t_1}^{t_2}f_n^2(s)|\tilde{u}(s) -u^{k,n}(s)|^2ds\right)^\frac{1}{2}\left(\int_{t_1}^{t_2}|k^{\varepsilon_{mi}}(s)-k(s)|^2ds\right)^\frac{1}{2}\\
&\leq C \left(\int_{t_1}^{t_2}f_n^2(s)|\tilde{u}(s) -u^{k,n}(s)|^2ds\right)^\frac{1}{2}.\\
\end{aligned}
\end{equation}
Combining (\ref{u1}) and (\ref{u2}) we infer that
\begin{equation}\label{l2}
\lim\limits_{i\rightarrow \infty}\sup\limits_{0\leq t\leq T}|\int_0^t f_n(s)\langle\tilde{u}(s) -u^{k,n}(s), \sigma(u^{k,n}(s))\rangle( k^{\varepsilon_{mi}}(s)-k(s))ds|=0.
\end{equation}
Thus, it follows from (\ref{l1}) and (\ref{l2}) that
\begin{equation}
\lim\limits_{i\rightarrow \infty} \sup\limits_{0\leq t\leq T}|\int_0^t f_n(s)\langle u^{k^{\varepsilon_{mi}},n}(s) -u^{k,n}(s), \sigma(u^{k,n}(s))\rangle( k^{\varepsilon_{mi}}(s)-k(s))ds|=0,
\end{equation}
completing the proof.


\section{Large Deviations}
\setcounter{equation}{0}
 \setcounter{definition}{0}
Recall that $u^\varepsilon$ is the solution of the reflected SEE:
\begin{equation}\label{ldp}    
\left\{
\begin{aligned}
du^\varepsilon(t)&+Au^\varepsilon(t)dt =f(u^\varepsilon(t))dt +B(u^\varepsilon(t),u^\varepsilon(t))dt + \sqrt{\varepsilon}\sigma(u^\varepsilon(t))dW(t)+dL^\varepsilon(t), t\geq0,\\
 u^\varepsilon(0)& =u_0, \quad u_0\in \overline{D},
\end{aligned}
\right.
\end{equation}
For $k\in K= L^2(0,T;H)$ and consider the deterministic reflected skeleton equation:
\begin{equation}\label{deter1} 
\left\{
\begin{aligned}
du^k(t)&+Au^k(t)dt =f(u^k(t))dt +B(u^k(t),u^k(t))dt + \sigma(u^k(t))k(t)dt+dL^k(t), t\geq0,\\
 u^k(0)& =u_0, \quad u_0\in \overline{D}.
\end{aligned}
\right.
\end{equation}
Define a mapping $\Gamma^0:C([0,T];\mathbb{R})\rightarrow X_T$ by
\begin{equation}
\Gamma^0\left(\int_0^\cdot k(s)ds\right):=u^k, \quad \text{For } k \in K,
\end{equation}
where $u^k$ is the solution of (\ref{deter1}). Here we introduce our main result:
\begin{theorem}\label{thldp}
Let the assumptions  \textbf{(A.1)-(A.3)} hold. Then the family $\{u^\varepsilon\}_{\varepsilon >0}$ satisfies a large deviation principle on the space $X_T$ with the rate function $I$ given by
\begin{equation}
I(g):=\inf\limits_{k\in K;g=\Gamma^0(\int_0^\cdot k(s)ds)}\left\{\frac{1}{2}\int_0^T k^2(s)ds\right\}, \quad g\in X_T,
\end{equation}
where the convention $\inf\{\emptyset\}=\infty$.
\end{theorem}

\begin{proof}
The existence of a unique strong solution of (\ref{ldp}) implies that for every $\varepsilon>0$, there exists a measurable mapping $\Gamma^{\varepsilon}(\cdot):C(0,T;\mathbb{R})\rightarrow X_T$ such that
\begin{equation}
u^{\varepsilon}=\Gamma^{\varepsilon}(W(\cdot)).
\end{equation}
To prove the theorem, it suffices to verify the conditions (i) and (ii) in Theorem \ref{cri}. And we can see that the condition (ii) is exactly the statement of Theorem \ref{4.1}. We are left to verify the condition (i) in Theorem \ref{cri}.\\
Let $\{k^\varepsilon,\varepsilon > 0\}\subset \widetilde{S}_N$ be a given family of stochastic processes. Applying Girsanov theorem it is easy to see that $u^{\varepsilon,k^\varepsilon}=\Gamma^{\varepsilon}\left(W(\cdot)+\frac{1}{\sqrt{\varepsilon}}\int_0^\cdot k^\varepsilon(s)ds\right)$ is the solution of following reflected SEE:
\begin{equation}\label{condition1}
\left\{
\begin{aligned}
d&u^{\varepsilon,k^\varepsilon}(t)+Au^{\varepsilon,k^\varepsilon}(t)dt =f(u^{\varepsilon,k^\varepsilon})dt +B(u^{\varepsilon,k^\varepsilon}(t),u^{\varepsilon,k^\varepsilon}(t))dt\\
&\hspace{5.3cm} + \sqrt{\varepsilon}\sigma(u^{\varepsilon,k^\varepsilon}(t))dW(t)+ \sigma(u^{\varepsilon,k^\varepsilon}(t))k^\varepsilon(t)dt + dL^{\varepsilon,k^\varepsilon}(t), t\geq0,\\
&u^{\varepsilon,k^\varepsilon}(0)=u_0, \quad u_0\in \overline{D}.
\end{aligned}
\right.
\end{equation}
\end{proof}
Moreover, $u^{k^\varepsilon}=\Gamma^0\left(\int_0^\cdot k^\varepsilon(s)ds\right)$ is the solution of reflected SEE:
\begin{equation}
\left\{
\begin{aligned}
du^{k^\varepsilon}&+Au^{k^\varepsilon}(t)dt =f(u^{k^\varepsilon}(t))dt +B(u^{k^\varepsilon}(t),u^{k^\varepsilon}(t))dt + \sigma(u^{k^\varepsilon}(t))k^\varepsilon(t)dt+dL^{k^\varepsilon}(t), t\geq0,\\
 u^k(0)& =u_0, \quad u_0\in \overline{D}.
\end{aligned}
\right.
\end{equation}
The condition (i) in Theorem \ref{cri} will be satisfied if we prove for any $N<\infty$, $\{k^\varepsilon,\varepsilon > 0\}\subset \widetilde{S}_N$ and any $\delta>0$
\begin{equation}\label{c2} 
\lim\limits_{\varepsilon\rightarrow 0}\mathbb{P}\left(|u^{\varepsilon,k^\varepsilon}-u^{k^\varepsilon}|_{X_T}>\delta \right)=0.
\end{equation}
Consider $g_\varepsilon(t)=exp(-\lambda\int_0^t\|u^{k^\varepsilon}(s)\|^2ds)$ for $\lambda>4$ and $t\geq 0$.\\
By $It\hat{o}'s$ formula, we have
\begin{equation}\label{Ivar}
\begin{aligned}
g_\varepsilon(t)|u^{\varepsilon,k^\varepsilon}(t) - u^{k^\varepsilon}(t)|^2\\
=& -\lambda\int_0^t g_\varepsilon(s)\|u^{k^\varepsilon}(s)\|^2|u^{\varepsilon,k^\varepsilon}(s) - u^{k^\varepsilon}(s)|^2 ds\\
 & -2\int_0^t g_\varepsilon(s)(u^{\varepsilon,k^\varepsilon}(s) - u^{k^\varepsilon}(s), A(u^{\varepsilon,k^\varepsilon}(s) - u^{k^\varepsilon}(s)))ds\\
 & +2\int_0^t g_\varepsilon(s)\langle u^{\varepsilon,k^\varepsilon}(s) - u^{k^\varepsilon}(s), \sigma(u^{\varepsilon,k^\varepsilon}(s)) -\sigma(u^{k^\varepsilon}(s))\rangle k^\varepsilon(s)ds\\
 & +2\sqrt{\varepsilon}\int_0^t g_\varepsilon(s)\langle u^{\varepsilon,k^\varepsilon}(s) - u^{k^\varepsilon}(s), \sigma(u^{\varepsilon,k^\varepsilon}(s))\rangle dWs\\
 & +2\int_0^t g_\varepsilon(s)\langle u^{\varepsilon,k^\varepsilon}(s) - u^{k^\varepsilon}(s), f(u^{\varepsilon,k^\varepsilon}(s)) -f(u^{k^\varepsilon}(s))\rangle ds\\
 & +2\int_0^t g_\varepsilon(s)\langle u^{\varepsilon,k^\varepsilon}(s) - u^{k^\varepsilon}(s), B(u^{\varepsilon,k^\varepsilon}(s)) -B(u^{k^\varepsilon}(s))\rangle ds\\
 & +2\int_0^t g_\varepsilon(s)\langle u^{\varepsilon,k^\varepsilon}(s) - u^{k^\varepsilon}(s), L^{\varepsilon,k^\varepsilon}(ds) -L^{k^\varepsilon}(ds)\rangle\\
 & +\varepsilon\int_0^t g_\varepsilon(s)|\sigma(u^{\varepsilon,k^\varepsilon}(s))|^2ds.\\
=&\ I^\varepsilon_1+I^\varepsilon_2+I^\varepsilon_3 +I^\varepsilon_4+I^\varepsilon_5+I^\varepsilon_6 +I^\varepsilon_7+I^\varepsilon_8.\\
\end{aligned}
\end{equation}
From$(\ref{def5})$ we see that $I^\varepsilon_7\leq 0$.
Similarly as in the proof of Proposition \ref{prop2}, we have
\begin{equation}
\begin{aligned}
I^\varepsilon_2&=-2\int_0^t g_\varepsilon(s)\|u^{\varepsilon,k^\varepsilon}(s) - u^{k^\varepsilon}(s)\|^2ds.\\
I^\varepsilon_3&\leq 2\int_0^t g_\varepsilon(s)|k^\varepsilon(s)||u^{\varepsilon,k^\varepsilon}(s) - u^{k^\varepsilon}(s)|^2ds.\\
I^\varepsilon_5&\leq \frac{1}{2}\int_0^t g_\varepsilon(s)\|u^{\varepsilon,k^\varepsilon}(s) - u^{k^\varepsilon}(s)\|^2ds+ C \int_0^t g_\varepsilon(s)|u^{\varepsilon,k^\varepsilon}(s) - u^{k^\varepsilon}(s)|^2ds.\\
I^\varepsilon_6&= 2\int_0^t g_\varepsilon(s)\langle u^{\varepsilon,k^\varepsilon}(s) - u^{k^\varepsilon}(s), B(u^{\varepsilon,k^\varepsilon}(s)-u^{k^\varepsilon}(s),(u^{k^\varepsilon}(s))\rangle ds\\
&\leq 4\int_0^t g_\varepsilon(s)\|u^{\varepsilon,k^\varepsilon}(s) - u^{k^\varepsilon}(s)\||u^{\varepsilon,k^\varepsilon}(s) - u^{k^\varepsilon}(s)|\|u^{k^\varepsilon}(s)\|ds\\
&\leq\int_0^t g_\varepsilon(s)\|u^{\varepsilon,k^\varepsilon}(s) - u^{k^\varepsilon}(s)\|^2ds+4 \int_0^t g_\varepsilon(s)|\|u^{k^\varepsilon}(s)\|^2u^{\varepsilon,k^\varepsilon}(s) - u^{k^\varepsilon}(s)|^2ds.
\end{aligned}
\end{equation}
Substituting the above estimates  back into (\ref{Ivar}),
\begin{equation}
\begin{aligned}
&\sup\limits_{0\leq t \leq T} g_\varepsilon(t)|u^{\varepsilon,k^\varepsilon}(t) - u^{k^\varepsilon}(t)|^2 + \frac{1}{2}\int_0^T g_\varepsilon(s)\|u^{\varepsilon,k^\varepsilon}(s) - u^{k^\varepsilon}(s)\|^2ds\\
&\leq C\int _0^T(1 +|k^\varepsilon(s)|)g_\varepsilon(s)\|u^{\varepsilon,k^\varepsilon}(s) - u^{k^\varepsilon}(s)\|^2ds\\
&\  +\sup\limits_{0\leq t \leq T} |2\sqrt{\varepsilon}\int_0^t g_\varepsilon(s)\langle u^{\varepsilon,k^\varepsilon}(s) - u^{k^\varepsilon}(s), \sigma(u^{\varepsilon,k^\varepsilon}(s))\rangle dWs|\\
&\  +\varepsilon\int_0^t g_\varepsilon(s)|\sigma(u^{\varepsilon,k^\varepsilon}(s))|^2ds.
\end{aligned}
\end{equation}
By the Gronwall's inequality it follows that
\begin{equation}
\begin{aligned}
&\sup\limits_{0\leq t \leq T} g_\varepsilon(t)|u^{\varepsilon,k^\varepsilon}(t) - u^{k^\varepsilon}(t)|^2 + \frac{1}{2}\int_0^T g_\varepsilon(s)\|u^{\varepsilon,k^\varepsilon}(s) - u^{k^\varepsilon}(s)\|^2ds\\
&\leq(M_1^\varepsilon+M_2^\varepsilon)exp\left(C\int_0^T(1 +|k^\varepsilon(s)|)ds\right)\leq C_{(N)}((M_1^\varepsilon+M_2^\varepsilon)),
\end{aligned}
\end{equation}
where
\begin{equation}
\begin{aligned}
&M_1^\varepsilon=\sup\limits_{0\leq t \leq T} |2\sqrt{\varepsilon}\int_0^t g_\varepsilon(s)\langle u^{\varepsilon,k^\varepsilon}(s) - u^{k^\varepsilon}(s), \sigma(u^{\varepsilon,k^\varepsilon}(s))\rangle dWs|\\,
&M_2^\varepsilon=\varepsilon\int_0^t g_\varepsilon(s)|\sigma(u^{\varepsilon,k^\varepsilon}(s))|^2ds.
\end{aligned}
\end{equation}
By Burkholder's inequality and using the fact that $|\sigma(u^{\varepsilon,k^\varepsilon}(s))|\leq C (1+|u^{\varepsilon,k^\varepsilon}(s)|)$, we see that
\begin{equation}
\begin{aligned}
E[M_1^\varepsilon]&\leq C\sqrt{\varepsilon} E [\int_0^T g_\varepsilon^2(s)| u^{\varepsilon,k^\varepsilon}(s) - u^{k^\varepsilon}(s)|^2(1+|u^{\varepsilon,k^\varepsilon}(s)|)^2ds]^\frac{1}{2}\\
&\leq C\sqrt{\varepsilon} E [\sup\limits_{0\leq t \leq T}(1+|u^{\varepsilon,k^\varepsilon}(s)|)\left(\int_0^T g_\varepsilon^2(s)| u^{\varepsilon,k^\varepsilon}(s) - u^{k^\varepsilon}(s)|^2ds\right)^\frac{1}{2}]\\
&\leq C\sqrt{\varepsilon} E [\sup\limits_{0\leq t \leq T}(1+|u^{\varepsilon,k^\varepsilon}(s)|)^2 +\int_0^T g_\varepsilon^2(s)| u^{\varepsilon,k^\varepsilon}(s) - u^{k^\varepsilon}(s)|^2ds]\\
&\rightarrow 0, \quad \text{as }\varepsilon\rightarrow 0.
\end{aligned}
\end{equation}
And it is clear that
\begin{equation}
E[M_2^\varepsilon]\leq \varepsilon C T E[\sup\limits_{0\leq t \leq T}(1+|u^{\varepsilon,k^\varepsilon}(s)|)]\rightarrow 0, \quad \text{as }\varepsilon\rightarrow 0.
\end{equation}
Hence,
\begin{equation}\label{p1}   
E[\sup\limits_{0\leq t \leq T} g_\varepsilon(t)|u^{\varepsilon,k^\varepsilon}(t) - u^{k^\varepsilon}(t)|^2] + E[\int_0^T g_\varepsilon(s)\|u^{\varepsilon,k^\varepsilon}(s) - u^{k^\varepsilon}(s)\|^2ds]\rightarrow 0, \quad \text{as }\varepsilon\rightarrow 0.
\end{equation}

Given $\delta>0$, for any $M>0$ we have
\begin{equation}\label{p} 
\begin{aligned}
&\hspace{-0.4cm}\mathbb{P}\left(|u^{\varepsilon,k^\varepsilon}-u^{k^\varepsilon}|_{X_T}>\delta \right)=\mathbb{P}\left(|u^{\varepsilon,k^\varepsilon}-u^{k^\varepsilon}|^2_{X_T}>\delta^2 \right)\\
&\hspace{-1cm}=\mathbb{P}\left(\sup\limits_{0\leq t \leq T} |u^{\varepsilon,k^\varepsilon}(t) - u^{k^\varepsilon}(t)|^2 + \int_0^T \|u^{\varepsilon,k^\varepsilon}(s) - u^{k^\varepsilon}(s)\|^2ds>\delta^2 \right)\\
&\hspace{-1cm}\leq \mathbb{P}\left(\sup\limits_{0\leq t \leq T} |u^{\varepsilon,k^\varepsilon}(t) - u^{k^\varepsilon}(t)|^2 + \int_0^T \|u^{\varepsilon,k^\varepsilon}(s) - u^{k^\varepsilon}(s)\|^2ds>\delta^2, \int_0^T\|u^{k^\varepsilon}(s)\|^2ds\leq M\right)\\
&\hspace{-1cm} \  +\mathbb{P}\left(\int_0^T\|u^{k^\varepsilon}(s)\|^2ds\geq M\right)\\
&\hspace{-1cm}\leq \mathbb{P}\left(\sup\limits_{0\leq t \leq T}g_\varepsilon(t) |u^{\varepsilon,k^\varepsilon}(t) - u^{k^\varepsilon}(t)|^2 + \int_0^T g_\varepsilon(s)\|u^{\varepsilon,k^\varepsilon}(s) - u^{k^\varepsilon}(s)\|^2ds>exp(-\lambda M)\delta^2, \int_0^T\|u^{k^\varepsilon}(s)\|^2ds\leq M\right)\\
&\hspace{-1cm} \  +\frac{1}{M}E[\int_0^T\|u^{k^\varepsilon}(s)\|^2ds]\\
&\hspace{-1cm}\leq exp(\lambda M)\frac{1}{\delta^2}E[\sup\limits_{0\leq t \leq T}g_\varepsilon(t) |u^{\varepsilon,k^\varepsilon}(t) - u^{k^\varepsilon}(t)|^2 + \int_0^T g_\varepsilon(s)\|u^{\varepsilon,k^\varepsilon}(s) - u^{k^\varepsilon}(s)\|^2ds]+\frac{1}{M}E[\int_0^T\|u^{k^\varepsilon}(s)\|^2ds].\\
\end{aligned}
\end{equation}
As $\{k^\varepsilon,\varepsilon > 0\}\subset \widetilde{S}_N$, we can easily show that $\sup\limits_\varepsilon E[\int_0^T\|u^{k^\varepsilon}(s)\|^2ds]<\infty$.
Now, for any $\eta>0$, we can choose $M>0$ such that
\begin{equation}
\frac{1}{M}E[\int_0^T\|u^{k^\varepsilon}(s)\|^2ds]\leq\eta, \quad \text {for any } \varepsilon >0.
\end{equation}
Then by letting $\varepsilon \rightarrow 0$ in (\ref{p}) and (\ref{p1}) we obtain
\begin{equation}\label{end}
\lim\limits_{\varepsilon\rightarrow 0}\mathbb{P}\left(|u^{\varepsilon,k^\varepsilon}-u^{k^\varepsilon}|_{X_T}>\delta \right)\leq \eta.
\end{equation}
As $\eta$ is arbitrary, assertion (\ref{c2}) follows from (\ref{end}).

\section{LDP of Reflected Stochastic Navier Stokes Equations}
\setcounter{equation}{0}
 \setcounter{definition}{0}
An important application of this paper is to treat the stochastic Navier Stokes Equations on a two dimensional domains $D$, which can be unbounded, e.g., the whole Euclidean space $\mathbb{R}^2$. We will concentrate on this case and mostly follow a recent paper \cite{BF}.

For a natural number $d$ and $p \in [1,\infty)$, let $L^p=L^p(\mathbb{R}^d, \mathbb{R}^d)$ be the classical Lebesgue space of all $\mathbb{R}^d$-valued lebesgue measurable functions $v=(v^1,\cdots,v^d)$ defined on $\mathbb{R}^d$ endowed with the following classical norm
\begin{equation}
\|v\|_{L^p}=\left(\sum_{k=1}^d\|v^k\|^p_{L^p(\mathbb{R}^d)}\right)^{\frac{1}{p}}.
\end{equation}
For $p=\infty$, we set $\|v\|_{L^\infty}=\text{max}_{k=1}^d\|v^k\|_{L^{\infty}(\mathbb{R}^d)}$.\\
Set $J^s=(I-\Delta)^{\frac{s}{2}}$. And define the generalized Sobolev spaces of divergence free vector distributions, for $s\in \mathbb{R}$, as
\begin{equation}
\begin{aligned}
H^{s,p}&=\{u\in \mathcal{S}'(\mathbb{R}^d,\mathbb{R}^d):\|J^s u\|_{L^p}<\infty\},\\
H^{s,p}_{\text{sol}}&=\{ u\in H^{s,p}: \text{div }u = 0\}.
\end{aligned}
\end{equation}
It's well-known that $J^\sigma$ is an isomorphism between $H^{s,p}$ and $H^{s-\sigma,p}$ for $s\in mathbb{R}$ and $ 1<p<\infty$. Moreover, $H^{s_2,p}\subset H^{s_1,p}$ when $s_1<s_2$. For the Hilbert case $p=2$ We set $H=H^{0,2}_{\text{sol}}$ and, for $s \neq 0,H^s=H^{0,2}_{\text{sol}}$, so that $H^s$ is a proper closed subspace of the classical Sobolev space usually denoted by the same symbol. In particular, we put
\begin{equation}
    H=\{v\in L^2(\mathbb{R}^d,\mathbb{R}^d)):\text{div }v = 0\}.
\end{equation}
with scalar product inherited from $L^2(\mathbb{R}^d, \mathbb{R}^d)$.

Denote by $\langle\cdot,\cdot\rangle$ the duality bracket between $(H^{s,p})'$ and $H^{s,p}$ spaces. Note that for $p \in [1,\infty)$, the space $(H^{s,p})'$ can be identified with $(H^{-s,p^{\ast}})$, where $\frac{1}{p}+\frac{1}{p^{\ast}}=1$.

Now we  define the operators appearing in the abstract formulation. Assume $s \in \mathbb{R}$ and $1\leq p <\infty$. Let $A_0=-\Delta$; then $A_0$ is a linear unbounded operator in $H^{s,p}$ and bounded from $H{s+2,p}$ to $H^{s,p}$. Moreover, the space $H^{s,p}_{\text{sol}}$ are invariant w.r.t. $A_0$ and the corresponding operator will be denoted by $A$. We can observe that $A$ is a linear unbounded operator in $H^{s,p}$ and bounded from $H^{s+2,p}_{\text{sol}}$ to $H^{s,p}_{\text{sol}}$.  The operator $-A_0$ generates a contractive and analytic $C-0$-semigroup $\{e^{-tA}\}_{t\geq 0}$ on $H^{s,p}$ and therefore, the operator $-A$ generates a contractive and analytic $C-0$-semigroup $\{e^{-tA}\}_{t\geq 0}$  on $H^{s,p}_{\text{sol}}$. Moreover, for $t > 0$ the operator$e^{-tA}$ is bounded from $H^{s,p}_{\text{sol}}$ into $H^{s',p}_{\text{sol}}$
with $s' > s$ and there exists a constant $M$ (depending on $s' - s$ and $p$) such that
\begin{equation}
    \|e^{-tA}\|_{\mathcal{L}({H^{s,p}_{\text{sol}},H^{s',p}_{\text{sol}}})}\leq M(1+t^{-(s'-s)2}).
\end{equation}
We have $A:H^1\rightarrow H^{-1}$ as a linear bounded operator and
\begin{equation}
    \langle Av,v\rangle =\|\nabla v\|^2_{L^2}, \quad v \in H^1,
\end{equation}
Where
\begin{equation}
    \|\nabla v\|^2_{L^2}=\sum_{k-1}^d \|\nabla v^k\|^2_{L^2}, \quad v \in H^1.\\
\end{equation}

Moreover,
\begin{equation}
    \|v\|^2_{H^1}=\|v\|^2_{L^2} + \|\nabla v\|^2_{L^2}.
\end{equation}

We define a bounded trilinear form $b:H^1\times H^1\times H^1 \rightarrow \mathbb{R}$ by
\begin{equation}
    b(u,v,z)=\int_{\mathbb{R}^d} (u(\xi)\cdot \nabla)v(\xi)\cdot z(\xi) d\xi, \quad u,v,z \in H^1.
\end{equation}
and the corresponding bounded bilinear operator $B:H^1\times H^1\times H^1 \rightarrow H^{-1}$ via the trilinear form
\begin{equation}
    \langle B(u,v),z\rangle = b(u,v,z), \quad u,v,z \in H^1.
\end{equation}
This operator satisfies, for all $u,v,z \in H^1$,
\begin{equation}
    \langle B(u,v),z\rangle=-\langle B(u,z),v\rangle, \quad \langle B(u,v),v\rangle =0.
\end{equation}
Finally, we define the noise  forcing term. Given a real separable Hilbert space $K$ we consider a $K$-cylindrical Wiener process $\{W(t)\}_{t\geq0}$ defined on a stochastic basis $(\Omega,\mathcal{F,\mathbb{F}}=(\mathcal{F}_t)_{t\geq 0},\mathbb{P})$ satisfying the usual conditions. For the covariance $\sigma$ of the noise we make the following assumptions.

\textbf{(G1)} the mapping $\sigma : H\rightarrow\gamma(K;H)$  is well-defined and is a Lipschitz continuous map $G : H \rightarrow \gamma(K;H)$, i.e.
\begin{equation}
    \exists L>0 :\|\sigma(v_1)-\sigma(v_2)\|_{\gamma(K;H)}\leq L\|v_1-v_2\|_H
\end{equation}
for all $v_1,v_2\in H$, where $\gamma(K,H)$ denote by the space of all Hilbert-Schmidt operators from $K$ to $H$ and by $\|\cdot\|_{ \gamma(K,H)}$ we denote the corresponding Hilbert-Schmidt norm.

We consider the stochastic damped Navier-Stokes equations, that is the equations of motion of a viscous incompressible fluid with two forcing terms, one is random and the other one is deterministic. These equations are
\begin{equation}\label {NV}
\left\{
\begin{aligned}
&\partial_t u + [-\nu \Delta u +\gamma u + (u\cdot \nabla)u + \nabla p] dt = \sigma(u) \partial_t W +f dt,\\
& \text{div} u =0,
\end{aligned}
\right.
\end{equation}
where the unknowns are the vector velocity $u = u(t,\xi)$ and the scalar pressure $p = p(t,\xi)$ for $t \geq 0$ and $\xi \in  \mathbb{R}^d$. By $\nu > 0$ we denote the kinematic viscosity and by $\gamma  > 0$ the sticky
viscosity. When $\gamma = 0$ (\ref{NV}) reduce to the classical stochastic
Navier-Stokes equations. The notation $\partial_t W$ on the right hand side is for the space correlated and white in time noise and $f$ is a deterministic forcing term. We consider a multiplicative term $\sigma(u)$ keeping track of the fact that the noise may depend on the velocity.
Projecting equations (\ref{NV}) onto the space $H$ of divergence free vector fields, we get the abstract form of the stochastic damped Navier-Stokes equations (\ref{NV}):
\begin{equation}\label{NV2}
    du(t) =[Au(t) +\gamma u(t) + B(u(t),u(t))]dt = \sigma(u(t)) dW(t) + f(t)dt
\end{equation}
with the initial condition $ u(0)=u_0,$ where the initial velocity $u_0:\Omega \rightarrow H$ is an $\mathcal{F}_0$-measurable random variable .  Here $\gamma >0 $ is fixed and for simplicity we have put $\nu =1$.

Now consider the reflected  stochastic Navier-Stokes equation:
\begin{equation}\label{RNV}
    du(t) =[Au(t) +\gamma u(t) + B(u(t),u(t))]dt = \sigma(u(t)) dW(t) +  f(t)dt +dL(t)
\end{equation}
with the initial condition $ u(0)=u_0\in \overline{D}$, and the skeleton equation:
\begin{equation}\label{RNV}
    du^k(t) =[Au^k(t) +\gamma u^k(t) + B(u^k(t),u^k(t))]dt = \sigma(u^k(t))k(t)dt +  f(t)dt +dL^k(t),
\end{equation}
for $k\in K=L^2([0, T], R)$.

Recall $\Gamma^0(\int_0^\cdot k(s)ds):=u^k$.
Let $u^\varepsilon$ be the solution of the following reflected problem:
\begin{equation}   
\left\{
\begin{aligned}
du^\varepsilon(t)&+Au^\varepsilon(t)dt =f(u^\varepsilon(t))dt +B(u^\varepsilon(t),u^\varepsilon(t))dt + \sqrt{\varepsilon}\sigma(u^\varepsilon(t))dW(t)+dL^\varepsilon(t), t\geq0,\\
 u^\varepsilon(0)& =u_0, \quad u_0\in \overline{D},
\end{aligned}
\right.
\end{equation}

As an application of the main result Theorem 5.1,
we have
\begin{theorem}
   The family $\{u^\varepsilon\}_{\varepsilon >0}$ satisfies a large deviation principle on the space $X_T$ with the rate function $I$ given by
\begin{equation}
I(g):=\inf\limits_{k\in K;g=\Gamma^0(\int_0^\cdot k(s)ds)}\left\{\frac{1}{2}\int_0^T k^2(s)ds\right\}, \quad g\in X_T,
\end{equation}
where the convention $\inf\{\emptyset\}=\infty$.
\end{theorem}

\vskip 0.5cm

\noindent{\bf Acknowledgement}.
This work is partially supported by National Key R$\&$D program of China (No. 2022 YFA1006001)), National Natural Science Foundation of China (Nos. 12131019, 12371151, 11721101).

\end{document}